\documentclass[a4paper]{article}

\usepackage{amsfonts,amsthm,amssymb,amsmath}
\usepackage[T1]{fontenc}
\usepackage{graphicx}
\usepackage{hyperref}
\usepackage{enumerate}

\newtheorem{theorem}{Theorem}
\newtheorem{lemma}[theorem]{Lemma}

\theoremstyle{remark}

\newcommand{\V}{\mathcal{V}}
\newcommand{\X}{\mathcal{X}}
\newcommand{\Y}{\mathcal{Y}}
\newcommand{\sub}{\subseteq}

\def\td{tree-decom\-po\-si\-tion}

\title{A short derivation of the structure theorem for graphs with excluded topological minors}
\author{Joshua Erde and Daniel Wei{\ss}auer}
\date{}

\begin{document}
\maketitle

\begin{abstract}
As a major step in their proof of Wagner's conjecture, Robertson and Seymour showed that every graph not containing a fixed graph $H$ as a minor has a tree-decomposition in which each torso is almost embeddable in a surface of bounded genus. Recently, Grohe and Marx proved a similar result for graphs not containing $H$ as a topological minor. They showed that every graph which does not contain $H$ as a topological minor has a tree-decomposition in which every torso is either almost embeddable in a surface of bounded genus, or has a bounded number of vertices of high degree. We give a short proof of the theorem of Grohe and Marx, improving their bounds on a number of the parameters involved.
\end{abstract}

\begin{section}{Introduction} \label{sec: introduction}

	A graph~$H$ is a \emph{minor} of a graph~$G$ if~$H$ can be obtained from a subgraph of~$G$ by contracting edges. In a series of 23 papers, published between 1983 and 2012, Robertson and Seymour developed a deep theory of graph minors which culminated in the proof of \emph{Wagner's Conjecture}~\cite{GMXX}, which asserts that in any infinite set of finite graphs there is one which is a minor of another. One of the landmark results proved along the way, and indeed a fundamental step in resolving Wagner's Conjecture, is a structure theorem for graphs excluding a fixed graph as a minor~\cite{GMXVI}. It is easy to see that~$G$ cannot contain~$H$ as a minor if there is a surface into which~$G$ can be embedded but~$H$ cannot. Loosely speaking, the structure theorem of Robertson and Seymour asserts an approximate converse to this, thereby revealing the deep connection between topological graph theory and the theory of graph minors:
	
	\begin{theorem}[\cite{GMXVI} (informal)] \label{t: RS}
		For any $n \in \mathbb{N}$, every graph excluding the complete graph~$K_n$ as a minor has a tree-decomposition in which every torso is almost embeddable into a surface into which $K_n$ is not embeddable.
	\end{theorem}
	
	A graph~$H$ is a \emph{topological minor} of a graph~$G$ if~$G$ contains a subdivision of~$H$ as a subgraph. It is easy to see that~$G$ then also contains~$H$ as a minor. The converse is not true, as there exist cubic graphs with arbitrarily large complete minors. For topological minors, we thus have an additional degree-based obstruction, which is fundamentally different from the topological obstruction of surface-embeddings for graph minors. Grohe and Marx~\cite{GM15} proved a result in a similar spirit to Theorem~\ref{t: RS} for graphs excluding a fixed graph as a topological minor:
	
	\begin{theorem}[\cite{GM15} (informal)] \label{t: GM}
		For any $n \in \mathbb{N}$, every graph excluding~$K_n$ as a topological minor has a tree-decomposition in which every torso either
		\begin{enumerate}[\rm (i)]
			\item has a bounded number of vertices of high degree, or
			\item is almost embeddable into a surface of bounded genus.
		\end{enumerate}
	\end{theorem}
 

 More recently, Dvo\v{r}\'{a}k~\cite{D12} refined the embeddability condition of this theorem to reflect more closely the topology of embeddings of an arbitrary graph~$H$ which is to be excluded as a topological minor.
 
 The proof given in~\cite{GM15}, which uses Theorem~\ref{t: GM} as a block-box, is algorithmic and explicitly provides a construction of the desired \td, however as a result the proof is quite technical in parts. In this paper, we give a short proof of Theorem~\ref{t: GM} which also provides a good heuristic for the structure of graphs without a large complete topological minor, as well as improving the implicit bounds given in \cite{GM15} on many of the parameters in their theorem. Our proof is non-constructive, but we note that it can easily be adapted to give an algorithm to find either a subdivision of $K_r$ or an appropriate \td. However, the run time of this algorithm will be much slower than that of the algorithm given in~\cite{GM15}.

	One of the fundamental structures we consider are $k$-blocks. A \emph{$k$-block} in a graph~$G$ is a set~$B$ of at least~$k$ vertices which is inclusion-maximal with the property that for every separation $(U,W)$ of order~$<\! k$, we either have $B \sub U$ or $B \sub W$. The notion of a $k$-block, which was first studied by Mader~\cite{mader78,mader74}, has previously been considered in the study of graph decompositions~\cite{CDHH13CanonicalParts, CDHH13CanonicalAlg,CG14:isolatingblocks}. 
    
    It is clear that a subdivision of a clique on~$k+1$ vertices yields a $k$-block. The converse is not true for any $k \geq 4$, as there exist planar graphs with arbitrarily large blocks. The second author~\cite{weissauer17block} proved a structure theorem for graphs without a $k$-block:
	
	\begin{theorem}[\cite{weissauer17block}] \label{t: W}
	Let~$G$ be a graph and $k \geq 2$. If $G$ has no $(k+1)$-block then $G$ has a tree-decomposition in which every torso has at most $k$ vertices of degree at least $2k(k-1)$.
	\end{theorem}

Now, since a subdivision of a complete graph gives rise to both a complete minor and a block, there are two obvious obstructions to the existence of a large topological minor, the absence of a large complete minor or the absence of a large block. The upshot of Theorem~\ref{t: GM} is that in a local sense these are the only obstructions, any graph without a large topological minor has a tree-decomposition into parts whose torsos either don't contain a large minor, or don't contain a large block. Furthermore, by Theorem~\ref{t: RS} and Theorem~\ref{t: W}, the converse should also be true: if we can decompose the graph into parts whose torsos either don't contain a large minor or don't contain a large block, then we can refine this tree-decomposition into one satisfying the requirements of Theorem~\ref{t: GM}.

The idea of our proof is as follows. Both large minors and large blocks point towards a `big side' of every separation of low order. A subdivision of a clique simultaneously gives rise to both a complete minor and a block and, what's more, the two are hard to separate in that they choose the same `big side' for every low-order separation. A qualitative converse to this is already implicit in previous work on graph minors and linkage problems: if a graph contains a large complete minor and a large block which cannot be separated from that minor, then the graph contains a subdivision of a complete graph. 

Therefore, if we assume our graph does not contain a subdivision of~$K_r$, then we can separate any large minor from every large block. It then follows from the \emph{tangle tree theorem} of Robertson and Seymour~\cite{GMX} ~-- or rather its extension to \emph{profiles}~\cite{profiles,ProfilesNew,CDHH13CanonicalAlg} ~-- that there exists a tree-decomposition which separates the blocks from the minors. Hence each part is either free of large minors or of large blocks. 

However, in order to apply Theorems~\ref{t: RS} and~\ref{t: W}, we need to have control over the \emph{torsos}, and not every tree-decomposition will provide that: it might be, for example, that separating some set of blocks created a large minor in one of the torsos. We therefore contract some parts of our tree-decomposition and use the minimality of the remaining separations to prove that this does not happen.

A second nice feature of our proof is that we avoid the difficulty of constructing such a tree-decomposition by choosing initially a tree-decomposition with certain connectivity properties, the proof of whose existence already exists in the literature, and then simply \emph{deducing} that this tree-decomposition has the required properties. 

We are going to prove the following:	
   
\begin{theorem} \label{main result}
	Let~$r$ be a positive integer and let~$G$ be a graph containing no subdivision of~$K_r$. Then~$G$ has a \td\ of adhesion~$<\! r^2$ such that every torso either
	\begin{enumerate}[\rm (i)]
		\item has fewer than~$r^2$ vertices of degree at least~$2r^4$, or
		\item has no $K_{2r^2}$-minor.
	\end{enumerate}
\end{theorem}		
	
	Combining Theorems~\ref{t: RS} and~\ref{main result} then yields Theorem~\ref{t: GM}. 
    
Let us briefly compare the bounds we get to the result of Grohe and Marx~\cite[Theorem 4.1]{GM15}. It is implicit in their results that if~$G$ contains no subdivision of~$K_r$, then~$G$ has a tree-decomposition of adhesion $O(r^6)$ such that every torso either has~$O(r^6)$ vertices of degree~$\Omega(r^7)$, has no $K_{\Omega(r^6)}$ minor or has size at most~$O(r^6)$. In this way, Theorem~\ref{main result} gives an improvement on the bounds for each of the parameters. Recently Liu and Thomas~\cite{LT14} also proved an extension of the work of Dvo\v{r}\'{a}k~\cite{D12}, with the aim to more closely control the bound on the degrees of the vertices in~(i). Their results, however, only give this structure `relative' to some tangle.

	\end{section}

     \begin{section}{Notation and background material} \label{sec: preliminaries}
     
	All graphs considered here are finite and undirected and contain neither loops nor parallel edges. Our notation and terminology mostly follow that of~\cite{DiestelBook16noEE}.     
	
	Given a tree~$T$ and $s, t \in V(T)$, we write $sTt$ for the unique $s$-$t$-path in~$T$. A \emph{separation} of a graph~$G=(V,E)$ is a pair $(A, B)$ with $V = A \cup B$ such that there are no edges between $A \setminus B$ and $B \setminus A$. The \emph{order} of $(A, B)$ is the number of vertices in $A \cap B$. We call the separation $(A,B)$ \emph{tight} if for all $x, y \in A \cap B$, both~$G[A]$ and~$G[B]$ contain an $x$-$y$-path with no internal vertices in $A \cap B$. 
	
	The set of all separations of~$G$ of order~$<\! k$ will be denoted by $S_k(G)$. An \emph{orientation} of~$S_k(G)$ is a subset of~$S_k(G)$ containing precisely one element from each pair $\{ (A,B), (B,A) \} \sub S_k(G)$. The orientation is \emph{consistent} if it does not contain two separations $(A, B), (C,D)$ with $B \sub C$ and $D \sub A$. A separation \emph{distinguishes} two orientations $O_1, O_2$ of~$S_k(G)$ if precisely one of $O_1, O_2$ contains it. It does so \emph{efficiently} if it has minimum order among all separations distinguishing them.
		
		Recall that, given an integer~$k$, a set~$B$ of at least~$k$ vertices of~$G$ is a \emph{$k$-block} if it is inclusion-maximal with the property that for every separation $(U,W)$ of \mbox{order~$<\! k$}, either $B \sub U$ or $B \sub W$. Observe that~$B$ induces a consistent orientation $O_B := \{ (U,W) \colon B \sub W \}$ of~$S_k(G)$.
		
		Given an integer~$m$, a \emph{model of~$K_m$} is a family~$\X$ of~$m$ pairwise disjoint sets of vertices of~$G$ such that $G[X]$ is connected for every $X \in \X$ and~$G$ has an edge between~$X$ and~$Y$ for any two $X, Y \in \X$. The elements of~$\X$ are called \emph{branch sets}. Note that, if $(U,W)$ is a separation of order~$<\! m$, then exactly one of $U \setminus W$ and $W \setminus U$ contains some branch set. In this way, $\X$ induces a consistent orientation~$O_{\X}$ of~$S_k(G)$, where $(U,W) \in O_{\X}$ if and only if some branch set of~$\X$ is contained in~$W$.
			
			A \emph{\td\ of~$G$} is a pair $(T, \V)$, where~$T$ is a tree and $\V = (V_t)_{t \in T}$ is a family of sets of vertices of~$G$ such that:
\begin{itemize}
\item for every $v \in V(G)$, the set of $t \in V(T)$ with $v \in V_t$ induces a non-empty subtree of~$T$;
\item for every edge $vw \in E(G)$ there is a $t \in V(T)$ with $v, w \in V_t$.
\end{itemize}
If $(T, \V)$ is a \td\ of~$G$, then every $st \in E(T)$ induces a separation
\[
(U_{s}, W_{t}) := ( \bigcup_{t \notin uTs} V_u, \bigcup_{s \not\in vTt} V_v).
\] 
Note that $U_s \cap W_t = V_s \cap V_t$. In this way, every edge $e \in E(T)$ has an \emph{order} given by the order of the separation it induces, which we will write as $|e|$. Similarly, an edge of~$T$ \emph{(efficiently) distinguishes} two orientations if the separation it induces does. We say that $(T, \V)$ \emph{(efficiently) distinguishes} two orientations~$O$ and~$P$ if some edge of~$T$ does. We call $(T, \V)$ \emph{tight} if every separation induced by an edge of~$T$ is tight.
			
		The \emph{adhesion} of $(T, \V)$ is the maximum order of an edge. If the adhesion of $(T, \V)$ is less than~$k$ and~$O$ is an orientation of $S_k(G)$, then~$O$ induces an orientation of the edges of~$T$ by orienting an edge~$st$ towards~$t$ if $(U_s, W_t) \in O$. If~$O$ is consistent, then all edges will be directed towards some node $t \in V(T)$, which we denote by~$t_O$ and call the \emph{home node} of~$O$. When~$O$ is induced by a block~$B$ or model~$\X$, we abbreviate $t_B := t_{O_B}$ and $t_{\X} := t_{O_{\X}}$, respectively. Observe that and edge $e \in E(T)$ distinguishes two orientations~$O$ and~$P$ if and only if $e \in E(t_OTt_P)$.
		
	Given $t \in V(T)$, the \emph{torso at~$t$} is the graph obtained from $G[V_t]$ by adding, for every neighbor~$s$ of~$t$, an edge between any two non-adjacent vertices in~$ V_s \cap V_t$. More generally, given a subtree $S \sub T$, the \emph{torso at~$S$} is the graph obtained from $G\left[\bigcup_{s \in S} V_s\right]$ by adding, for every edge $st \in E(T)$ with $S \cap \{ s, t \} = \{ s \}$, an edge between any two non-adjacent vertices in $V_s \cap V_t$.
				 
	We also define contractions on tree-decompositions: Given $(T,\V)$ and an edge $st \in E(T)$, to contract the edge~$st$ we form a \td\ $(T',\V')$ where
\begin{itemize}
\item $T'$ is obtained by contracting~$st$ in~$T$ to a new vertex~$x$;
\item Let $V_x' := V_s \cup V_t$ and $V'_u := V_u$ for all $u \in V(T) \setminus \{ s, t \}$.
\end{itemize}
It is simple to check that $(T', \V')$ is a \td. We note that the separations induced by an edge in $E(T) \setminus \{ st \}$ remain the same, as do the torsos of parts $V_u$ for $u \neq s,t$.
	
We say a \td\ $(T, \V)$ is \emph{$k$-lean} if it has adhesion~$<\! k$ and the following holds for all $p \in [k]$ and $s, t \in T$: If $sTt$ contains no edge of order~$<\! p$, then every separation $(A,B)$ with $|A \cap V_s| \geq p$ and $|B \cap V_t| \geq p$ has order at least~$p$.
	 
Let $n := |G|$. The \emph{fatness} of $(T, \V)$ is the sequence $(a_0, \ldots, a_n)$, where~$a_i$ denotes the number of parts of order $n - i$. A \td\ of lexicographically minimum fatness among all \td s of adhesion smaller than~$k$ is called \emph{$k$-atomic}. These \td s play a pivotal role in our proof, but we actually only require two properties that follow from this definition. It was observed by Carmesin, Diestel, Hamann and Hundertmark~\cite{ForcingBlocks} that the short proof of Thomas' Theorem~\cite{thomas90} given by Bellenbaum and Diestel in~\cite{bellenbaumDiestel} also shows that $k$-atomic \td s are $k$-lean (see also \cite{GJ16}).
    
    \begin{lemma}[\cite{bellenbaumDiestel}]\label{lean lemma}
    Every $k$-atomic tree-decomposition is $k$-lean.
    \end{lemma}
    
	It is also not hard to see that $k$-atomic \td s are tight. In~\cite{weissauer17block}, the second author used $k$-atomic \td s to prove a structure theorem for graphs without a $k$-block. In fact, the proof given there yields the following:

    \begin{lemma}[\cite{weissauer17block}] \label{structure thm blocks}
  	Let~$G$ be a graph and~$k$ a positive integer. Let~$(T, \V)$ be a $k$-atomic \td\ of~$G$ and $t \in V(T)$ such that~$V_t$ contains no $k$-block of~$G$. Then the torso at~$t$ contains fewer than~$k$ vertices of degree at least~$2k^2$.
    \end{lemma}

     Let~$G$ be a graph and $Z \sub V(G)$. We denote by~$G^Z$ the graph obtained from~$G$ by making the vertices of~$Z$ pairwise adjacent. A \emph{$Z$-based model} is a model~$\X$ of~$K_{|Z|}$ such that $X \cap Z$ consists of a single vertex for every $X \in \X$.
     
     The following lemma of Robertson and Seymour~\cite{GMXIII} is crucial to our proof.
     
          \begin{lemma}[\cite{GMXIII}] \label{robertson seymour lemma}
     	     		Let~$G$ be a graph, $Z \sub V(G)$ and $p := |Z|$. Let $q \geq 2p-1$ and let~$\X$ be a model of~$K_q$ in~$G^Z$. If~$\X$ and~$Z$ induce the same orientation of $S_p(G^Z)$, then~$G$ has a $Z$-based model.
     \end{lemma}

     \end{section}
     
	\begin{section}{The proof} \label{sec: proof}
    
    Let us fix throughout this section a graph~$G$ with no subdivision of~$K_r$, let $k:= r(r-1)$, $m := 2k$, and let $(T, \V)$ be a $k$-atomic \td\ of~$G$.
	
	 First, we will show that $(T, \V)$ efficiently distinguishes every $k$-block from every model of~$K_m$ in~$G$. This allows us to split $T$ into two types of sub-trees, those containing a $k$-block and those containing a model of~$K_m$. Lemma~\ref{structure thm blocks} allows us to bound the number of high degree degree vertices in the torsos in the latter components. We will then show that if we choose these sub-trees in a sensible way then we can also bound the order of a complete minor contained in the torsos of the former. Hence, by contracting each of these sub-trees in $(T, \V)$ we will have our desired \td.
		
		To show that $(T, \V)$ distinguishes every $k$-block from every model of~$K_m$ in~$G$, we must first show that they are distinguishable, that is, no $k$-block and~ $K_m$ induce the same orientation. The following lemma, as well as its proof, is similar to Lemma~6.11 in~\cite{GM15}.
		
		   \begin{lemma} \label{block + minor = topclique}
     	Let~$B$ be a $k$-block and~$\X$ a model of~$K_m$ in~$G$. If~$B$ and~$\X$ induce the same orientation of~$S_{k}$, then~$G$ contains a subdivision of~$K_r$ with arbitrarily prescribed branch vertices in~$B$.
     \end{lemma}
     
     \begin{proof}
		Suppose~$B$ and~$\X$ induce the same orientation and let $B_0$ be an arbitrary subset of $B$ of size~$r$. Let~$H$ be the graph obtained from~$G$ by replacing every $b \in B_0$ by an independent set~$J_b$ of order~$(r-1)$, where every vertex of~$J_b$ is adjacent to every neighbor of~$b$ in~$G$ and to every vertex of~$J_c$ if $b, c$ are adjacent. Let $J := \bigcup_b J_b$ and note that $|J| = k$. We regard~$G$ as a subgraph of~$H$ by identifying each $b\in B$ with one arbitrary vertex in~$J_b$. In this way we can regard~$\X$ as a model of~$K_m$ in~$H$.
      
        Assume for a contradiction that there was a separation $(U,W)$ of~$H$ such that $|U \cap W| < |J|$, $J \sub U$ and $X \sub W \setminus U$ for some $X \in \X$. We may assume without loss of generality that for every $b \in B_0$, either $J_b \sub U \cap W$ or $J_b \cap (U \cap W) = \emptyset$. Indeed, if there is a $z \in J_b \setminus (U \cap W)$, then $z \in U \setminus W$, and we can delete any $z' \in J_b \cap W$ from~$W$ and maintain a separation (because $N(z) = N(z')$) with the desired properties. In particular, for every $b \in B_0$ we find $b \in W$ if and only if $J_b \sub W$. Since $|U \cap W| < |J|$, it follows that there is at least one $b_0 \in B_0$ with $J_{b_0} \sub (U \setminus W)$. Let $(U', W') := (U \cap V(G), W \cap V(G))$ be the induced separation of~$G$. Then $X \sub W' \setminus U'$ and $b_0 \in U' \setminus W'$. Since $|U' \cap W'| \leq |U \cap W| < k$ and~$B$ is a $k$-block, we have $B \sub U'$. But then $(U',W')$ distinguishes~$B$ and~$\X$, which is a contradiction to our initial assumption.
        
        We can now apply Lemma~\ref{robertson seymour lemma} to~$H$ and find a $J$-based model $\Y = (Y_j)_{j \in J}$ in~$H$. For each $b\in B_0$, label the vertices of~$J_b$ as $(v^b_c)_{c \in B_0 \setminus \{ b \}}$. For $b \neq c$, $H$ has a path $P_{b,c}' \sub Y_{v^b_c} \cup Y_{v^c_b}$ and the paths obtained like this are pairwise disjoint, because the~$Y_j$ are, and $P'_{b,c} \cap J = \{ v^b_c, v^c_b\}$. For each such path $P_{b,c}'$, obtain $P_{b,c} \sub G$ by replacing~$v^b_c$ by~$b$ and~$v^c_b$ by~$c$. The collection of these paths $(P_{b,c})_{b,c \in B_0}$ gives a subdivision of~$K_r$ with branch vertices in~$B_0$.
	\end{proof}
    
    Now we can show that $(T,\mathcal{V})$ efficiently distinguishes every $k$-block from every model of~$K_m$ in~$G$.

		\begin{lemma} \label{treedec eff dist}
$(T,\mathcal{V})$ efficiently distinguishes all orientations of~$S_k(G)$ induced by $k$-blocks or models of~$K_m$.
\end{lemma}

	\begin{proof}
	Let us call a consistent orientation~$O$ of~$S_k(G)$ \emph{anchored} if for every $(U,W) \in O$, there are at least~$k$ vertices in $W \cap V_{t_O}$.
		
	Note that every orientation $O = O_B$ induced by a $k$-block~$B$ is trivially anchored, since $B \sub V_{t_B}$. But the same is true for the orientation $O = O_{\X}$ induced by a model~$\X$ of~$K_m$. Indeed, let $(U, W) \in O_{\X}$. Then every set in~$\X$ meets~$V_{t_{\X}}$. At least~$k$ branch sets of~$\X$ are disjoint from $U \cap W$, say $X_1, \ldots, X_k$, and they all lie in $W \setminus U$. For $1 \leq i \leq k$, let $x_i \in X_i \cap V_{t_{\X}}$ and note that $R := \{ x_1, \ldots, x_k\} \sub W \cap V_{t_{\X}}$.
	
	We now show that $(T, \V)$ efficiently distinguishes all anchored orientations of~$S_k(G)$. Let $O_1, O_2$ be anchored orientations of~$S_k(G)$ and let their home nodes be $t_1$ and $t_2$ respectively. If $t_1 \neq t_{2}$, let~$p$ be the minimum order of an edge along $t_1Tt_{2}$, and put $p := k$ otherwise. Choose some $(U, W) \in O_2 \setminus O_1$ of minimum order. Since~$O_1$ and~$O_2$ are anchored, we have $|U \cap V_{t_1}| \geq k$ and $|W \cap V_{t_2}| \geq k$. As $(T , \V)$ is $k$-lean, it follows that $|U \cap W| \geq p$. Hence $t_1 \neq t_2$ and $(T, \V)$ efficiently distinguishes~$O_1$ and~$O_2$.
\end{proof}		

Let us call a node $t \in V(T)$ a \emph{block-node} if it is the home node of some $k$-block and \emph{model-node} if it is the home node of a model of~$K_m$. 

Let $F \sub E(T)$ be inclusion-minimal such that every $k$-block is efficiently distinguished from every model of~$K_m$ by some separation induced by an edge in~$F$. We now define a red/blue colouring $c: V(T) \rightarrow \{r,b\}$ by letting $c(t) = b$ if the component of $T-F$ containing~$t$ contains a block-node and letting $c(t)=r$ if it contains a model-node. Let us first show that this is in fact a colouring of $V(T)$.
    
    \begin{lemma} \label{colouring total}
    	Every node receives exactly one colour.
    \end{lemma} 
    
    \begin{proof}
    	Suppose first that $t \in V(T)$ is such that the component of $T-F$ containing~$t$ contains both a block node and a model node. Then there is a $k$-block~$B$ and a $K_m$-minor~$\X$ such that $t_BTt$ and $t_{\X}Tt$ both contain no edges of~$F$. But then~$B$ and~$\X$ are not separated by the separations induced by~$F$, a contradiction.
    	
    	Suppose now that $t \in V(T)$ is such that the component $S$ of $T-F$ containing~$t$ contains neither a block nor a minor. Let $f_1, \ldots, f_n$ be the edges of~$T$ between~$S$ and $T \setminus S$, ordered such that $|f_1| \geq |f_i|$ for all $i \leq n$. By minimality of~$F$, there is a block-node~$t_B$ and a model-node~$t_{\X}$ such that~$f_1$ is the only edge of~$F$ that efficiently distinguishes~$B$ and~$\X$. Since $t_B, t_{\X} \notin S$, there is a $j \geq 2$ such that $f_j \in E(t_BTt_{\X})$, and so~$f_j$ distinguishes~$B$ and~$\X$ as well, and since $|f_1| \geq |f_j|$, it does so efficiently, contradicting our choice of~$B$ and~$\X$
    \end{proof}
    
	   \begin{lemma} \label{join separators to minors}
    Let $st \in E(T)$ and suppose~$s$ is blue and~$t$ is red. Then~$G[W_t]$ has a $(V_s \cap V_t)$-based model.
    \end{lemma}
    
        \begin{proof}
  Let $Q := V_s \cap V_t$. Let~$t_B$ be a block-node in the same component of $T - F$ as~$s$ and let~$t_{\X}$ be a model-node in the same component as~$t$. Since the separations induced by~$F$ efficiently distinguish~$B$ and~$\X$, it must be that $st \in F$ and $(U_s, W_t)$ efficiently distinguishes~$B$ and~$\X$.
    
  Let $\Y := (X \cap W_t)_{X \in \X}$. Since $(U_s, W_t) \in O_{\X}$, $\Y$ is a model of~$K_m$ in $G[W_t]^Q$. We wish to apply Lemma~\ref{robertson seymour lemma} to~$Q$ and~$\Y$ in the graph~$G[W_t]$.
    Suppose $Q$ and $\Y$ do not induce the same orientation of $S_{|Q|}(G[W_t]^Q)$. That is, there is a separation $(U,W)$ of~$G[W_t]^Q$ with $|U \cap W| < |Q|$ and $Q \sub U$ such that $Y\cap U = \emptyset$ for some $Y \in \Y$. There is an $X \in \X$ so that $Y = X \cap G[W_t]$. Note that $X \cap U$ is empty as well.
    Now $(U', W') := (U \cup U_s, W)$ is a separation of~$G$. Note that
    \[
	X \cap U' = X \cap U_s = \emptyset ,
    \]
    because~$X$ is connected, meets~$W_t$ and does not meet~$Q$. Therefore $X \sub W' \setminus U'$ and $B \sub U_s \sub U'$. But $|U' \cap W'| = |U \cap W| < |Q|$, which contradicts the fact that $(U_s, W_t)$ efficiently distinguishes~$B$ and~$\X$. Therefore, by Lemma~\ref{robertson seymour lemma}, $G[W_t]$ has a $Q$-based model.
    \end{proof}

Using the above we can bound the size of a complete minor in the torso of a blue component. The next lemma plays a similar role to Lemma~6.9 in~\cite{GM15}.
	    \begin{lemma} \label{torso minor-free}
    Let $S \sub T$ be a maximal subtree consisting of blue nodes.	Then the torso of~$S$ has no $K_m$-minor.
    \end{lemma}
    
       \begin{proof}
    Let $F_S := \{ (s,t) \colon st \in E(T), s \in S, t \notin S\}$. For every $(s,t) \in F_S$, the node~$s$ is blue and~$t$ is red. By Lemma~\ref{join separators to minors}, $G_t$ has a $(V_s \cap V_t)$-based complete minor~$\Y^{s,t}$. Contract each of its branch sets onto the single vertex of $V_s \cap V_t$ that it contains. Do this for every $(s,t) \in F_S$. After deleting any vertices outside of $V_S := \bigcup_{s \in S} V_s$, we obtain the torso of~$S$ as a minor of the graph~$G$.
    
    Suppose the torso of~$S$ contained a $K_m$-minor. Then~$G$ has a $K_m$-minor~$\X$ such that every $X \in \X$ meets~$V_S$. Therefore~$\X$ orients every edge $st \in E(T)$ with $(s,t) \in F_S$ towards~$s$. But then $t_{\X} \in S$, contradicting the assumption that~$S$ contains no red nodes.    
    \end{proof}	

 We can now finish the proof. Let $(T',\V')$ be obtained from $(T,\V)$ by contracting every maximal subtree consisting of blue nodes and let the vertices of~$T'$ inherit the colouring from $V(T)$. We claim that $(T',\V')$ satisfies the conditions of Theorem \ref{main result}.
 
Indeed, firstly, the adhesion of $(T',\V')$ is at most that of $(T,\V)$, and hence is at most $k$. Secondly, the torso of every red node in $(T',\V')$ is the torso of some red node in $(T, \V)$, which by Lemma~\ref{structure thm blocks} has fewer than~$k$ vertices of degree at least~$2k^2$. Finally, by Lemma~\ref{torso minor-free} the torso of every blue node in $(T',\V')$ has no $K_m$ minor. Since $k=r(r-1)$ and $m=2k$, the theorem follows.\\

As claimed in the introduction, it is not hard to turn this proof into an algorithm to find either a subdivision of~$K_r$ or an appropriate tree-decomposition. Indeed, the proof of Lemma~\ref{lean lemma} can easily be adapted to give an algorithm to find a tight $k$-lean tree-decomposition. Similarly, in order to colour the vertices of the tree red or blue we must check for the existence of a~$K_m$ minor or a $k$-block having this vertex as a home node, both of which can be done algorithmically (see~\cite{GMXIII} and~\cite{ForcingBlocks}). However, we note that the running time of such an algorithm, or at least a naive implementation of one, would have run time $\sim |V(G)|^{f(r)}$ for some function of the size of the topological minor~$K_r$ we are excluding, whereas the algorithm of Grohe and Marx has run time $g(r)|V(G)|^{O(1)}$, which should be much better for large values of~$r$.

	\end{section}

\bibliographystyle{plain}
\bibliography{topminbib} 

\end{document}